\documentclass[11pt, letterpaper,reqno]{amsart}

\usepackage{graphicx}
\usepackage[linktocpage=true,colorlinks,citecolor=blue,linkcolor=blue,urlcolor=black]{hyperref}
\usepackage{fullpage}
\usepackage{amssymb}
\usepackage{cite}
\usepackage{amsmath}
\usepackage{latexsym}
\usepackage{amscd}
\usepackage{tikz}
\usepackage{amsthm}
\usepackage{mathrsfs}
\usepackage{url}
\usepackage[utf8]{inputenc}
\usepackage[english]{babel}
\usepackage{amsfonts}
\usepackage{todonotes}
\usepackage{mathtools}
\usepackage{verbatim}

\numberwithin{equation}{section}
\def\HH{\mathcal H}
\def\NN{\mathbb N}
\def\rl{R_\ell}

\def\eps{\varepsilon}
\def\ex{\textnormal{ex}}

\theoremstyle{plain}
\newtheorem{theorem}{Theorem}
\newtheorem{thm}[theorem]{Theorem}

\newtheorem{lem}[theorem]{Lemma}

\newtheorem{prop}[theorem]{Proposition}
\newtheorem*{theorem*}{Theorem}

\theoremstyle{remark}

\theoremstyle{definition}

\theoremstyle{remark}

\numberwithin{equation}{section}

\begin{document}
\title{Multicolor list Ramsey numbers grow exponentially}
\author{Jacob Fox}
\author{Xiaoyu He}
\author{Sammy Luo}
\author{Max Wenqiang Xu}

\thanks{Department of Mathematics, Stanford University, Stanford, CA 94305, USA. \\
\indent Fox is supported by a Packard Fellowship and by NSF award DMS-1855635. Email: {\tt jacobfox@stanford.edu}. \\
\indent He is supported by NSF GRFP Grant DGE-1656518. Email: {\tt alkjash@stanford.edu}.\\
\indent Luo is supported by NSF GRFP Grant DGE-1656518. Email: {\tt sammyluo@stanford.edu}.\\
\indent Xu is supported by the Cuthbert C. Hurd Graduate Fellowship, Stanford. Email: {\tt maxxu@stanford.edu}.}
	\maketitle

\begin{abstract}
    The list Ramsey number $R_{\ell}(H,k)$, recently introduced by Alon, Buci\'c, Kalvari, Kuperwasser, and Szab\'o, is a list-coloring variant of the classical Ramsey number. They showed that if $H$ is a fixed $r$-uniform hypergraph that is not $r$-partite and the number of colors $k$ goes to infinity, $e^{\Omega(\sqrt{k})} \le R_{\ell} (H,k) \le e^{O(k)}$. We prove that  $R_{\ell}(H,k) = e^{\Theta(k)}$ if and only if $H$ is not $r$-partite.
\end{abstract}

\section{Introduction}

The $k$-color \emph{Ramsey number} $R(H,k)$ of an $r$-uniform hypergraph $H$ (henceforth {\it $r$-graph}) is the smallest $n$ such that any $k$-coloring of the edges of the complete $r$-graph $K_n^{(r)}$ on $n$ vertices contains a monochromatic copy of $H$.
One of the oldest problems in Ramsey theory is to determine the growth rate of $R(K_3,k)$ in terms of $k$. In showing the existence of nontrivial modular solutions to the Fermat equation, Schur~\cite{Schur} in 1916 implicitly proved that
\[
\Omega(2^k)= R(K_3,k)=O(k!).\]
See \cite{Nesetril} for a detailed discussion of the history of this result. 
While the exponential constant in the lower bound has since been improved, the upper bound has only been improved by a constant factor. Whether $R(K_3, k)$ grows exponentially or faster is a famous open problem. Among other connections, it is related to the study of Shannon capacity~\cite{EMT}.

In analogy with the well-studied list-coloring version of the chromatic number, a variant of $R(H,k)$ called the list Ramsey number was recently introduced on MathOverflow \cite{BOF} and by Alon, Buci\'c, Kalvari, Kuperwasser, and Szab\'o~\cite{ABKKS2019List}.
Let $L:E(K_n^{(r)})\to \binom{\NN}{k}$ assign a set of $k$ colors to each edge of $K_n^{(r)}$. An \emph{$L$-coloring of $K_n^{(r)}$} is an edge-coloring where each edge $e$ is given a color in $L(e)$. The $k$-color {\it list Ramsey number} $\rl(H,k)$ of an $r$-graph $H$ is defined as the smallest $n$ such that there is some $L:E(K_n^{(r)})\to \binom{\NN}{k}$ for which every $L$-coloring of $K_n^{(r)}$ contains a monochromatic copy of $H$. 

Taking $L$ to be constant across all edges, we see that $\rl(H,k)\le R(H,k).$ While no exponential upper bound is known for $R(H,k)$, the authors of \cite{ABKKS2019List} proved the striking Theorem \ref{thm:abkks-upper-lower} below, showing that $\rl(H,k)$ is at most exponential in $k$. 

Let $H$ be an $r$-graph. The Tur\'an density $\pi(H)$ of $H$ is defined as $\lim_{n\rightarrow \infty} \ex(n,H)\binom{n}{r}^{-1}$, where $\ex(n,H)$ is the maximum number of edges in an $H$-free $r$-graph on $n$ vertices. The chromatic number $ \chi(H)$ is the smallest integer $k$ such that there is a $k$-vertex-coloring of $H$ without monochromatic edges. If $H$ has at least two edges, 
let \[
m(H):= \max_{H'\subseteq H,e(H')>1}\frac{e(H')-1}{v(H')-r}.
\]

\begin{thm}[{Alon, Buci\'c, Kalvari, Kuperwasser, and Szab\'o \cite[Theorems 5 and 6]{ABKKS2019List}}]
\label{thm:abkks-upper-lower} If $H$ is an $r$-graph that is not $r$-partite, then, as $k$ tends to infinity, we have
\[ e^{\sqrt{k\log (\chi(H)-1)/(4r)}} \le \rl(H,k)\le(1-\pi(H)+o(1))^{-km(H)}.
\]
\end{thm}
Erd\H{o}s \cite{Erdos64} proved that an $r$-graph $H$ satisfies $\pi(H)=0$ if and only if $H$ is $r$-partite. If $H$ is $r$-partite, it is known that $R(H,k)$, and hence $\rl(H,k)$, grows polynomially in $k$. Thus, Theorem~\ref{thm:abkks-upper-lower} shows that $\rl(H,k)$ always grows at most exponentially in $k$ regardless of uniformity. In stark contrast, the same upper bound is not even known for the classical Ramsey number $R(K_3,k)$. For fixed $r$, $R(K_{r+1}^{(r)},k)$ grows at least as fast as a tower function in $k$ of height $r$ (see \cite{EHR}). This shows that the list Ramsey number can be much smaller than the corresponding classical Ramsey number.

The authors of \cite{ABKKS2019List} asked which side of Theorem~\ref{thm:abkks-upper-lower} is closer to the truth. Our main result shows that the exponential upper bound gives  the correct order of growth.

\begin{thm}
\label{thm:general-lower}If $H$ is an $r$-graph that is not $r$-partite, then for any $k\ge1$,
\[
\rl(H,k) \ge c_{r}\cdot(1-\pi(H))^{-k/(r-1)},
\]
where $c_{r}=((r-2)!/e)^{1/(r-1)}$.
\end{thm}

Together with the upper bound in Theorem~\ref{thm:abkks-upper-lower}, this shows that $\rl(H,k)$ is exponential
in $k$ if and only if $\pi(H)>0$ (i.e., if and only if $H$ is not $r$-partite), and polynomial in $k$
otherwise. We remark that since $m(H) > \frac{1}{r-1}$ unless $H$ is $r$-partite, the exponential constant $m(H)$ in Theorem~\ref{thm:abkks-upper-lower} is always larger than the exponential constant $\frac{1}{r-1}$ in Theorem~\ref{thm:general-lower}, so our results do not determine $\lim_{k\rightarrow\infty} \rl(H,k)^{1/k}$, or even show that this limit exists. It would be interesting to determine this limit, even in just the case $H=K_3$, where Theorems~\ref{thm:abkks-upper-lower} and~\ref{thm:general-lower} give $\frac{1}{e} \cdot 2^k \le  \rl(K_3,k)\le (4+o(1))^k$. If one could show $\rl(K_3,k) \le 3.199^k$, comparing to the best known lower bound~\cite{XiZhExRa} for $R(K_3,k)$ would yield $\rl(K_3,k) < R(K_3,k)$.

The heart of our proof of Theorem~\ref{thm:general-lower} is a simple probabilistic argument. To illustrate the main idea, we demonstrate a weaker version of the argument in the case $H=K_3$.

\begin{prop}
\label{prop:trianglelb}
For any integer $k\ge 1$,
\[
\rl(K_3,k)> 2^{k/2}.
\]
\end{prop}
\begin{proof}
Let $n= 2^{k/2}$. We show that for every $L:E(K_n)\to \binom{\NN}{k}$, there is an $L$-coloring of the edges of $K_n$ with no monochromatic triangle. For each color $c\in \NN$, independently pick a complete spanning balanced bipartite graph $G_c$ on the vertices of $K_n$ uniformly at random. 

For each edge $e$, define $B_e$ to be the event that no color $c\in L(e)$ satisfies $e\in E(G_c)$. Since $\Pr[e\in E(G_c)]>\frac{1}{2}$ for each $c$, we have $\Pr[B_e]<2^{-k}$. There are $\binom{n}{2} < 2^k$ total edges in $E(K_n)$, so the union bound shows that with positive probability, none of the $B_e$ hold. In this case, for each edge $e$, pick a color $c\in L(e)$ such that $e \in E(G_c)$. Thus, there exists an $L$-coloring of $K_n$ where the edges in any given color $c$ form a bipartite graph. In particular, this $L$-coloring has no monochromatic triangles, as desired.
\end{proof}

This argument generalizes to give exponential lower bounds for $R_\ell(H,k)$ for any $r$-graph $H$ that is not $r$-partite. Instead of choosing $G_c$ to be a complete balanced bipartite graph, we take an $H$-free $r$-graph $G$ on $n$ vertices and choose each $G_c$ to be a copy of $G$ with the vertices permuted uniformly at random. Fixing $G$ with the maximum possible number of edges $\ex(n,H)$, we obtain the following lower bound.
\begin{prop}
\label{prop:weaklb}
If $H$ is an $r$-graph and $n < \big(1-\ex(n,H)/\binom{n}{r} \big)^{-k/r},$ then $\rl(H,k) > n$.
\end{prop}
Since $\ex(n,H)\binom{n}{r}^{-1}\to \pi(H)$ as $n$ grows large, this bound is already comparable to our final goal Theorem~\ref{thm:general-lower}, but with an inferior exponential constant.

Thus far, we have dealt with the case $\pi(H)>0$, when $H$ is not $r$-partite. In the case $\pi(H)=0$, it is proven in \cite[Theorem 8]{ABKKS2019List} that 
\begin{equation}\label{eq:bounds-r-partite}
   R(H, \lfloor ck/\log k\rfloor)\le R_\ell(H,k)\le R(H,k) 
\end{equation}
 for a constant $c=c(H)>0$. Assuming $\ex(n,H) = \Theta(n^{r - \eps(H)})$ for some constant $\eps(H)>0$, Proposition~\ref{prop:weaklb} implies that $R_\ell(H,k) = \Omega((k/\log k)^{1/\eps(H)})$, which implies the lower bound in (\ref{eq:bounds-r-partite}), and may be better by a polylogarithmic factor in some cases.

In Section~\ref{sec:generallb}, we prove Theorem~\ref{thm:general-lower} by strengthening the argument in the proof of Proposition~\ref{prop:trianglelb}, applying the Lov\'asz Local Lemma instead of using a union bound. A few technical complications arise for the sake of optimizing the constant $c_r$ and generalizing to hypergraphs. Afterwards, in Section~\ref{sec:conclusion}, we give some further applications of our method to other variants of Ramsey numbers. We systematically omit floor and ceiling signs whenever they are not crucial.

\section{Proof of Theorem~\ref{thm:general-lower}}

\label{sec:generallb}

We will use the following variant of the symmetric Lov\' asz Local Lemma; see \cite[Corollary 5.1.2]{AlSp} and the discussion afterwards. Unlike the standard version, we do not require full independence between events but only an upper bound on the relevant conditional probability.

\begin{lem}\label{lem:LLL}
Let $A_1,\ldots, A_n$ be events in an arbitrary probability space. Suppose there exists a directed graph $D= (V,E)$ on vertex set $V=[n]$ with maximum out-degree $d$ and for each $1\le i \le n$,
\begin{equation}\label{eq:LLL-dependence}
\Pr\Big[A_i | \bigwedge_{(i,j) \not \in E} \overline{A_j}\Big] \le p.
\end{equation}
If $ep(d+1)\le 1$, then $\Pr[\wedge_{i=1}^n \overline{A_i}] > 0$.
\end{lem}

For an $r$-graph $G$, the minimum degree of $G$ is denoted by $\delta(G)$, the number of vertices of $G$ is denoted by $v(G)$, and the number of edges of $G$ is denoted by $e(G)$.
For $r$-graphs $H$ and $G$, we say that $G$ is \emph{$H$-homomorphism-free} if there is no homomorphism from $H$ to $G$.

\begin{lem}
\label{lem:lower}If $H,G$ are $r$-graphs such that $G$ is $H$-homomorphism-free, then for any $k\ge1$,
\[
\rl(H,k)>c_{r}\cdot\left(1-\frac{(r-1)!\cdot \delta(G)}{v(G)^{r-1}}\right)^{-k/(r-1)},
\]
where $c_{r}=((r-2)!/e)^{1/(r-1)}$.
\end{lem}

\begin{proof}
Let 
\[
n=c_{r}\cdot\left(1-\frac{(r-1)!\cdot \delta(G)}{v(G)^{r-1}}\right)^{-k/(r-1)}.
\]
Our goal is to show $\rl(H,k)> n$. That is, given any list assignment
$L:E(K_{n}^{(r)})\rightarrow\binom{\NN}{k}$, we would like to
show there is an $L$-edge-coloring of $K_{n}^{(r)}$ with no monochromatic
copy of $H$. To each $c\in\NN$, assign a uniform random map $\phi_{c}:V(K_{n}^{(r)})\rightarrow V(G)$, write $\phi_c(e) \coloneqq \{\phi_c(u) : u\in e\}$,
and define for each $e\in E(K_{n}^{(r)})$ the restricted list $L'(e)\coloneqq \{c\in L(e):\phi_{c}(e)\in E(G)\}$. Let $B_e$ be the event $L'(e) = \emptyset$. We will apply Lemma~\ref{lem:LLL} to the events $B_e$ to show that with positive probability, no $B_e$ occurs.

For each edge $e\in E(K_n^{(r)})$, arbitrarily distinguish a single vertex $u_e\in e$, and let $D$ be the directed graph on vertex set $V(D)\coloneqq E(K_n^{(r)})$ such that $(e,e') \in E(D)$ if and only if $e \cap e' \not\subseteq \{ u_e\}$. We claim that the conditions of Lemma~\ref{lem:LLL} are satisfied for the family of events $\{B_e\}_{e\in E(K_n^{(r)})}$ with this directed graph $D$,
\[
    d \coloneqq \binom{n}{r} - \binom{n-r+1}{r} - 1 \textnormal{, and } p \coloneqq \left(1-\frac{(r-1)!\cdot \delta(G)}{v(G)^{r-1}}\right)^{k}.
\]
Since $D$ is $d$-regular, the fact that $D$ has maximum out-degree $d$ is clear. It remains to check condition (\ref{eq:LLL-dependence}) holds for the family of events $B_e$, i.e. that 
\begin{equation}\label{eq:LLL-condition-B}
\Pr\Big[B_e | \bigwedge_{(e,e') \not \in E(D)} \overline{B_{e'}}\Big] \le p.
\end{equation}

By the definition of $D$, if $(e,e')\not \in E(D)$ then $B_{e'}$ only depends on the values of $\phi_c(u)$ for $u \not \in e\setminus \{u_e\}$. Thus, to prove (\ref{eq:LLL-condition-B}), it suffices to show that for any fixed choices of the values of $\phi_c(u)$ for all $c\in \NN$ and $u\not \in e\setminus \{u_e\}$, the conditional probability of $B_e$ is at most $p$. Furthermore, $B_e$ is mutually independent of all the random variables $\phi_c(u)$ with $u\not \in e$, so we only need to condition on the choices of $(\phi_c(u_e))_{c\in \NN}$.

For each $c\in \NN$, fix an arbitrary $v_c\in V(G)$. We claim that
\begin{equation}\label{eq:degree-lower}
\Pr[c \in L'(e) | c \in L(e) \textnormal{ and } \phi_c(u_e) = v_c] \ge \frac{(r-1)!\cdot \delta(G)}{v(G)^{r-1}}.
\end{equation}
To see this, observe that the choices of $\phi_c(w)$ are independent of those of $\phi_c(u_e)$ for each $w\in e\setminus\{u_e\}$. Since there are at least $\delta(G)$ edges of $G$ containing $u_e$, there are at least $(r-1)!\cdot \delta(G)$ total choices of the $(r-1)$-tuple $(\phi_c(w))_{w\in e\setminus\{u_e\}}$ for which $\phi_c(e)$ forms an edge of $E(G)$. Since the values of this tuple are chosen uniformly at random out of $V(G)^{r-1}$, we obtain inequality (\ref{eq:degree-lower}).

Since the events for distinct $c$ are independent, and $B_{e}$ is the event $L'(e)=\emptyset$, we thus have
\begin{equation}\label{eq: deduction}
\Pr[B_{e}|\forall c\in \NN, \phi_c(u_e)=v_c]\le \left(1-\frac{(r-1)!\cdot \delta(G)}{v(G)^{r-1}}\right)^{k} = p
\end{equation}
for each $e$, conditional on any choice of $v_c\in V(G)$ for each color $c$. This proves inequality (\ref{eq:LLL-condition-B}). Finally,
with our choices of $n$, $p$, and $d$, we have
\[
ep(d+1) \le e \left(1-\frac{(r-1)!\cdot \delta(G)}{v(G)^{r-1}}\right)^{k} \cdot \frac{n^{r-1}}{(r-2)!} \le 1,
\]
and so the conditions of Lemma~\ref{lem:LLL} are satisfied, and we find that with positive probability, all $B_e$ do not hold. In other words, there exists some choice of the maps $\phi_{c}$ such that $L'(e)\ne\emptyset$
for all $e$. Arbitrarily color each $e$ by some
$c\in L'(e)\subseteq L(e)$ to obtain an $L$-coloring of $K_{n}^{(r)}$
where the edges in any given color $c$ form a subgraph which has
a homomorphism to $G$. Since $G$ is $H$-homomorphism-free, this $L$-coloring has no monochromatic copy of $H$, completing the proof.
\end{proof}
It remains to maximize the quantity $(r-1)!\delta(G)/v(G)^{r-1}$ over all $r$-graphs
$G$ that are $H$-homomorphism-free. We will prove the following lemma using a standard supersaturation result combined with Zykov's symmetrization trick \cite{Zykov}.

\begin{lem}\label{lem:extremal-mindeg}
For each $\eps>0$ and $r$-graph $H$ which has at least one edge, there exists an $r$-graph $G$ which is $H$-homomorphism-free and
\[
\delta(G) \ge (\pi(H)-\eps)\cdot \frac{v(G)^{r-1}}{(r-1)!}.
\]
\end{lem}
\begin{proof}
We may assume $\pi(H) \geq \eps$ as otherwise we may pick $G$ to be an empty $r$-graph. 

Define $\ex_{\hom}(H,n)$ to be the maximum number of edges in an $r$-graph $G$ on $n$ vertices which is $H$-homomorphism-free, and define $\pi_{\hom}(H)\coloneqq\lim_{n\rightarrow\infty}\binom{n}{r}^{-1}\ex_{\hom}(H,n)$. As shown in 
\cite{Keevash} using a supersaturation argument, we have $\pi(H)=\pi_{\hom}(H)$ for all $r$-graphs $H$.

In particular, for $n$ sufficiently large there exists an $r$-graph $G$ on $n$ vertices  which is $H$-homomorphism-free and 
\[
e(G)\ge \left(\pi(H) - \frac{\eps}{3}\right) \binom{n}{r} \ge \left(\pi(H) -\frac{2\eps}{3}\right) \frac{n^r}{r!}.
\]
Among all such $r$-graphs $G$ with $n$ vertices, pick a $G$ maximizing $e(G)$. We claim that 
\[
\delta(G) \ge \delta \coloneqq (\pi(H)-\eps)\cdot \frac{n^{r-1}}{(r-1)!}.
\]
Suppose this is false, and there exists a vertex $v\in V(G)$ with degree smaller than $\delta$. Let $u$ be a vertex of maximum degree in $G$. Consider the $r$-graph $G'$ obtained by removing $v$ and inserting a copy $u'$ of $u$. The degree of $u$ in $G$ is at least $r \cdot e(G)/n \ge \delta + \frac{\eps}{3} n^{r-1}/(r-1)!$ and the number of edges containing both $u$ and $v$ is at most $\binom{n}{r-2} < \frac{\eps}{3} n^{r-1}/(r-1)!$ for $n$ sufficiently large, so the degree of $u$ in $G\setminus \{v\}$ is strictly larger than $\delta$. It follows that $e(G') > e(G)$, as more edges are added than removed.

As $G'$ is obtained from a subgraph of $G$ by copying a vertex, we see that $G'$ also is $H$-homomorphism-free, contradicting the maximality of $G$. Hence $\delta(G)\ge \delta$, as desired.
\end{proof}

We now complete the proof of Theorem \ref{thm:general-lower}.
\begin{proof}[Proof of Theorem \ref{thm:general-lower}.]
 By Lemma \ref{lem:extremal-mindeg}, for every $\eps>0$ there exists
an $r$-graph $G$ which is $H$-homomorphism-free and 
\[
\delta(G) \ge (\pi(H)-\eps)\cdot \frac{v(G)^{r-1}}{(r-1)!}.
\]
Applying Lemma~\ref{lem:lower} to this pair of $H, G$ gives $\rl(H,k)>c_{r}\cdot\left(1-\pi(H)+\eps\right)^{-k/(r-1)}.$ Taking $\eps \to 0$, we obtain $\rl(H,k)\ge c_{r}\cdot\left(1-\pi(H)\right)^{-k/(r-1)}$ as desired.
\end{proof}

\section{Concluding remarks}
\label{sec:conclusion}
We conclude with a brief discussion of two related problems to which our methods can be applied: list Ramsey numbers for families of graphs, and the list variants of size Ramsey numbers and degree Ramsey numbers. These are also interesting questions for hypergraphs, but we restrict our attention to the graph case ($r=2$) in this section for simplicity.
\subsection{List Ramsey numbers for families of graphs}
While Theorem~\ref{thm:general-lower} does not determine the exponential constant in $\rl(H,k)$ for any particular $H$, we can determine $\rl(\HH,k)$ up to a constant factor for certain {\it families} $\HH$. Here, if $\HH$ is a family of graphs, $\rl(\HH,k)$ is the minimum $n$ such that there exists an $L:E(K_n)\to \binom{\NN}{k}$ for which every $L$-coloring of $K_n$ contains a monochromatic copy of some $H\in \HH$. 

\begin{thm}\label{thm:families}
If $s\ge 2$ and $\HH_s$ is the family of graphs with chromatic number greater than $s$, then
\[
\frac{1}{e} \cdot s^k \le \rl(\HH_s, k) \le s^k + 1.
\]
\end{thm}

We omit the details of the proof. The lower bound follows from the fact that our construction for Theorem~\ref{thm:general-lower} to avoid a monochromatic $H$ also avoids every graph with chromatic number at least $\chi(H)$. The upper bound follows from the fact that $\rl(\HH_s, k) \le R(\HH_s, k) = s^k + 1$. We conjecture that the upper bound in Theorem~\ref{thm:families} is the truth.

\subsection{Size and degree variants of list Ramsey numbers }

Our lower bounds generalize naturally to the settings of size and degree Ramsey numbers, which are well-studied variants of the classical Ramsey number introduced in \cite{EFRS} and \cite{BEL} respectively. Recall that while the classical Ramsey number $R(H,k)$ is the minimum number of \emph{vertices} in a graph $G$ for which any $k$-edge-coloring has a monochromatic copy of $H$, the size Ramsey number of $H$ is the minimum number of \emph{edges} in such a graph $G$, and the degree Ramsey number is the minimum max-degree of such a graph $G$. The growth 
of size Ramsey numbers and degree Ramsey numbers of bounded degree graphs have been of much interest \cite{CFS2015,Kang,KMW,KRSS,Jiang,RoSz,Tait}. 

We say a graph $G$ is {\it $k$-color list Ramsey for $H$} if there exists $L:E(G)\to \binom{\NN}{k}$ for which every $L$-coloring of the edges of $G$ contains a monochromatic copy of $H$. The \emph{$k$-color list size Ramsey number} of $H$, denoted $R_{\ell,e}(H,k)$, is the minimum number of edges of a graph $G$ which is $k$-color list Ramsey for $H$. 
By comparing the definitions, we have $R_{\ell, e}(H,k)\le \binom{R_{\ell}(H,k)}{2}$. We remark that the proof of Proposition~\ref{prop:weaklb} can be adapted to show that $R_{\ell,e}(H,k) \ge (1-\pi(H))^{-k}$. Indeed, if $G$ is a graph with $n$ vertices and $m$ edges with 
\[
m < (1-\pi(H))^{-k} \le \left(1-\ex(n,H)/\binom{n}{2} \right)^{-k},
\]
then $G$ is not $k$-color list Ramsey for $H$.

The \emph{$k$-color list degree Ramsey number} of $H$, denoted $R_{\ell, d}(H,k)$, is the minimum $\Delta$ for which there is a graph $G$ of maximum degree $\Delta$ which is $k$-color list Ramsey for $H$.
The proof of Theorem~\ref{thm:general-lower} can be adapted to show that $R_{\ell,d}(H,k)=\Omega((1-\pi(H))^{-k})$ for every graph $H$ with at least one edge. Indeed, if $G$ has maximum degree $\Delta \le c (1-\pi(H))^{-k}$ for an appropriate $c=c(H)>0$, then the proof of Lemma~\ref{lem:lower} proceeds similarly with the dependency digraph having maximum degree $d\le 2\Delta$, noting that any two bad events $B_e, B_{e'}$ are independent if $e$ and $e'$ are disjoint. Thus $G$ is not list Ramsey for $H$. This generalizes to the case of $r$-graphs, yielding the same bound. 
\vspace{3mm}

\noindent {\bf Acknowledgments.} The authors are grateful to József Balogh, Matija Buci\'c and Yuval Wigderson for helpful conversations, and the anonymous referees for helpful comments.

\end{document}